\documentclass[10pt,a4paper]{article}
\usepackage{amsmath}
\usepackage{amsfonts}
\usepackage{amssymb}
\usepackage{amsthm}
\usepackage{relsize}
\usepackage{geometry}

\newtheorem{thm}{Theorem}
\newtheorem{lem}[thm]{Lemma}

\newtheorem*{thmn}{Theorem}

\theoremstyle{definition}
\newtheorem{defn}[thm]{Definition}

\theoremstyle{remark}

\newcommand{\Alt}{\mathrm{Alt}}

\newcommand{\bN}{\mathbb{N}}

\begin{document}

\title{On a construction of A. Lucchini}

\author{Colin D. Reid\\
Universit\'{e} catholique de Louvain\\
Institut de Recherche en Math\'{e}matiques et Physique (IRMP)\\
Chemin du Cyclotron 2, 1348 Louvain-la-Neuve\\
Belgium\\
colin@reidit.net}

\maketitle

\begin{abstract}We show that the group constructed in the 2004 paper `A 2-generated just-infinite profinite
group which is not positively finitely generated' by A. Lucchini is in fact hereditarily just infinite and contains every countably based profinite group as a closed subgroup.  Lucchini's paper predates some related work of J. Wilson.\end{abstract}


\begin{thmn}\label{mainthm}Let $G$ be the profinite group constructed in \cite{Luc}.  Then $G$ is a hereditarily just infinite profinite group into which every countably based profinite group can be embedded as a closed subgroup.\end{thmn}

Lucchini's construction predates a paper by J. Wilson (\cite{Wil}) on `large' hereditarily just infinite groups.  In particular, Lucchini's group turns out to be an early example of the groups asserted to exist in Theorem A of \cite{Wil}.  We will give the proof as a series of lemmas.

\begin{defn}\label{typdef}Say a profinite group $G$ is of \emph{Lucchini type} if it is the inverse limit of a sequence of finite groups as follows:

$G_1$ is any finite group.  Thereafter $G_{i+1} = A_i \wr_{\Omega_i} G_i$ (with the natural surjection $G_{i+1} \rightarrow G_i$), such that:

$\Omega_i$ is a free (but not necessarily transitive) $G_i$-set;

$A_i$ is a non-abelian finite simple group;

every finite group appears as a subgroup of $A_i$ for infinitely many $i$.

Given $G$ of this form, write $T_i$ for the base group of $A_i \wr_{\Omega_i} G_i$.\end{defn}

\begin{lem}\label{lucgp}Let $G$ be the profinite group constructed in \cite{Luc}.  Then $G$ is of Lucchini type.\end{lem}

\begin{proof}Lucchini's group is constructed as the inverse limit of the sequence $\{X_i\}_{i \in \bN}$, where $X_1 = \Alt(5)$.  For each $i \ge 1$, one first takes the group $A_i \wr X_i$ where the wreathing action is regular and $A_i = \Alt(|X_i|)$.  Note that $|X_i| \ge 5$ so $A_i$ is a non-abelian simple group.  Write $S_i$ for the base group of $A_i \wr X_i$. Then $X_{i+1}$ is constructed as a subdirect power of $A_i \wr X_i$:
\[ X_{i+1} = \{ (g_1,\dots,g_{t_i}) \mid g_j \in A_i \wr X_i, \; S_ig_1 =S_ig_2 = \dots = S_ig_{t_i} \}.\]
(The value of $t_i$ is not important for the present discussion.)

Note that $X_{i+1}$ can be regarded as an intransitive wreath product $A_i \wr_{\Omega_i} X_i$, where $\Omega_i$ is an $X_i$-set consisting of $t_i$ free orbits of $X_i$.  Finally, since $|X_i| \rightarrow \infty$ it is clear that any finite group appears as a subgroup of $A_i = \Alt(|X_i|)$ for all but finitely many $i$.  Thus $G$ is of Lucchini type as required.\end{proof}

From now on, $G$ will be a group of Lucchini type, with finite groups as in Definition \ref{typdef}, and homomorphisms $\pi_i: G \rightarrow G_i$ arising from the inverse limit.

\begin{lem}\label{hered}$G$ has a descending chain $G > N_1 > N_2 > \dots$ of open normal subgroups with trivial intersection, such that each $N_i$ is of Lucchini type.\end{lem}

\begin{proof}Let $G_i$ be a sequence of finite groups of the prescribed form, and let $\pi_i: G \rightarrow G_i$ be the homomorphism arising from the inverse limit.  Let $N_j = \ker \pi_j$.  Clearly $G > N_1 > N_2 > \dots$ is a descending chain of open normal subgroups with trivial intersection.  Fix $j$ and set $H_i = \pi_{i+j}(N_j)$.  Then for all $i\ge 1$, $H_{i+1} = T_{i+j} \wr_{\Omega_{i+j}} H_i$.  Moreover $H_i \le G_{i+j}$, so $H_i$ acts freely on $\Omega_{i+j}$.  The condition on the set of groups $\{A_i \mid i \ge 1\}$ is clearly inherited by the set $\{A_{i+j} \mid i \ge 1\}$.  Hence $N_j$ is of Lucchini type.\end{proof}

\begin{lem}\label{hji}All groups of Lucchini type are hereditarily just infinite.\end{lem}

\begin{proof}Let $G$ be a group of Lucchini type.  By Lemma \ref{hered} it suffices to show that $G$ is just infinite.  Let $N$ be a non-trivial closed normal subgroup of $G$.  Then there exists $j$ such that for all $i \ge j$, $\pi_i(N)$ is non-trivial.  In this case $\pi_i(N)$ acts freely and non-trivially on $\Omega_i$, so that $\pi_i(N)$ regarded as a subgroup of $T_i \rtimes G_i$ does not centralise any of the simple factors of $T_i$.  It follows that $[\pi_i(N),T_i] = T_i$, so $\pi_{i+1}(N) \ge T_i$.  As this applies for all $i \ge j$, in fact $N \ge \ker \pi_j$, so $N$ is an open subgroup of $G$.\end{proof}

\begin{lem}\label{univ}Let $G$ be a group of Lucchini type.  Then every countably based profinite group can be embedded into $G$ as a closed subgroup.\end{lem}

\begin{proof}Let $H$ be a countably based profinite group.  Then $H$ is the inverse limit of a countable collection $\{H_i \mid i \in \bN\}$ of finite groups, and so $H$ is isomorphic to a closed subgroup of $\prod_{i \in \bN} H_i$.  Furthermore, $\prod_{i \in \bN} H_i$ is isomorphic to a closed subgroup of $\prod_{i \in \bN} A_i$: choose $i_1$ such that $H_1$ embeds in $A_{i_1}$, then $i_2 > i_1$ such that $H_2$ embeds in $A_{i_2}$, and so on.  Finally, $\prod_{i \in \bN} A_i$ appears naturally as a closed subgroup of $G$: take $A_i$ to be the diagonal subgroup of the direct power $T_i$, and note that these copies of $A_i$ all appear as subgroups of $G$ such that $[A_i,A_j]=1$ for all $i \not= j$.\end{proof}

\section*{Acknowledgments}My thanks go to Laurent Bartholdi for persuading me to look again at \cite{Luc} and determine whether the group in question is hereditarily just infinite.

\end{document}